\documentclass[12pt]{amsart}

\usepackage{amsmath}
\usepackage{amsfonts}
\usepackage{amssymb,enumerate}
\usepackage{amsthm}
\usepackage{fancyhdr}
\usepackage{tikz}
\usepackage{verbatim}
\usetikzlibrary{arrows,matrix}
\usepackage{hyperref}

\usepackage{caption}
\usepackage[noend]{algpseudocode}
\usepackage{epstopdf}

\newtheorem{lemma}{Lemma}

\newtheorem{theorem}[lemma]{Theorem}
\newtheorem{question}[lemma]{Question}

\title[]{On sparse perfect powers}



\author{A. Moscariello}
\address{Dipartimento di Matematica, Università di Pisa, Largo Bruno Pontecorvo 5, 56127 Pisa, Italy.}

\email{moscariello@mail.dm.unipi.it}

\begin{document}
\maketitle
\begin{abstract}
This work is devoted to proving that, given an integer $x \ge 2$, there are infinitely many perfect powers, coprime with $x$, having exactly $k \ge 3$ non-zero digits in their base $x$ representation, except for the case $x=2, k=4$, for which a known finiteness result by Corvaja and Zannier holds.
\end{abstract}
\section*{Introduction}

Let $k$ and $x$ be positive integers, with $x \ge 2$. In this work, we will study perfect powers having exactly $k$ non-zero digits in their representation in a given basis $x$. These perfect powers are exactly (up to dividing by a suitable factor) the set solutions of the Diophantine equation
\begin{equation}\label{maincase}
y^d=c_0+\sum_{i=1}^{k-1} c_ix^{m_i},
\end{equation} 
with $y,d$ positive integers greater than $1$, and $c_0,c_1,\dots,c_{k-1} \in \{1,\dots,x-1\}$ and $m_1 < \dots < m_{k-1}$ positive integers. 
We call perfect powers having a fixed number of non-zero digits \emph{sparse}, borrowing the terminology used for polynomials (a \emph{sparse} polynomial is a polynomial having \emph{relatively few} non-zero terms, compared to its degree)
Special cases of this innocent problem has been widely studied in the literature, and its appearence is quite deceiving: for instance, the lowest case, obtained with the positions $k=2$, $c_0=c_1=1$, is the well-known Catalan's conjecture, first proposed in 1844, which stood open for nearly 150 years before being proved by Mihailescu (cf. \cite{M1}) in the case $x=2$. Furthermore, the case $k=2$, $x > 2$ (i.e. perfect powers having exactly two digits in their base $x > 2$ representation) is still open (cf. \cite[\S 4.4.3]{N}), and is related to the well-known ABC conjecture. 

This class of problems also presents some ties to algebraic geometry. In fact,  Corvaja and Zannier showed in \cite{CZ2} that solutions of an equation of the form (\ref{maincase}) are associated with $S$-integral points on certain projective varieties. For instance, assume for the sake of simplicity that $x=p$ is a prime number, and that $k,d$ are fixed and $c_0=c_1=\dots=c_{k-1}=1$ in equation (\ref{maincase}). Consider, in the projective space $\mathbb{P}_k$, the variety $\mathbb{P}_k \setminus D$, where $D$ denotes the divisor consisting of the $k-1$ lines $X_i=0$, for $i=0,\dots,k-2$, and the hypersurface $X_{k-1}^d=X_0^d+\displaystyle \sum_{i=1}^{k-2} X_0^{d-1}X_i$, and let $S=\{\infty, p\}$. Then, $S$-integral points of this variety are such that the values $y_i=\frac{X_i}{X_0}$, where $i=1,\dots,k-2$, and $y_{k-1}=\left( \frac{X_{k-1}}{X_0} \right)^d - 1 - \displaystyle \sum_{i=1}^{k-2} \frac{X_i}{X_0}$ are all $S$-units. Also, the elements $y_i$ all have the form $\pm p^{m_i}$ and are such that $1+y_1+\dots+y_{k-1}$ is a $d$th perfect power, and are thus solutions of equation (\ref{maincase}). Now, the study of these points, and their distribution, can also be seen as a particular instance of a conjecture by Lang and Vojta (see \cite{HS}); in our context, this conjecture would imply that the set of $S$-integral points on $\mathbb{P}_k \setminus D$ is not Zariski dense.

Besides Mihailescu's Theorem, the more general case $k=2$ is still open; however, there is some evidence suggesting that there may be only a finite number of perfect powers having exactly two non-zero digits in any given base $x$. The case $k=3$ has been studied recently (cf. \cite{BBM1}, \cite{CZ4}); in particular, Corvaja and Zannier developed in  \cite{CZ4} an approach using $v$-adic convergence of analytic series at $S$-unit points to reduce this problem to the study of polynomial identities involving lacunary polynomial powers (i.e. polynomial powers $P(T)^d$ having a fixed number $k$ of terms). This method allowed them to provide a classification of perfect powers having exactly three non-zero digits. 

Specifically, for $x=2$ they obtained the following	characterization.

\begin{theorem}[\protect{\cite{CZ1}}]\label{3digits}
	For $d \ge 2$ integer, the perfect $d$th powers in $\mathbb{N}$ having at most three non-zero digits in the binary scale form the union of finitely many sets of the shape $\{q2^{md} \ | \ m \in \mathbb{N}\}$ and, if $d=2$, also the set $\{(2^a+2^b)^2 \ | \ a,b \in \mathbb{N}\}$.
\end{theorem} 
In the same work, the authors comment that their method can be used to obtain results equivalent to Theorem \ref{3digits} for any given base $x$. Actually, Theorem \ref{3digits} states that if $k=3$, $x=2$ there are only a finite number of \emph{exceptional} solutions, and the infinite family $y=(2^a+1)$, $d=2$, corresponding to the polynomial identity $(T+1)^2=T^2+2T+1$. 

Intuitively, one might expect that as the number of terms $k$ increases, the number of polynomial powers $P(T)^d$ having exactly $k$ terms increases as well. Moreover, since Corvaja and Zannier's method can be adjusted to study perfect powers with $k \ge 3$ non-zero digits, under certain assumption, we might infer that there is an increasing number of infinite families of solutions to equation (\ref{maincase}).

However, this is not necessarily the case. In fact, while studying the case $k=4$, Corvaja and Zannier obtained families of lacunary polynomial powers having exactly $4$ terms that are not related to solutions of the Diophantine equation $y^d=c_0+c_12^{m_1}+c_22^{m_2}+c_32^{m_3}$. Actually, they proved that this Diophantine equation has only finitely many solutions.
\begin{theorem}[\protect{\cite[Theorem 1.1]{CZ1}}]\label{4digits}
	There are only finitely many odd perfect powers in $\mathbb{N}$ having precisely four non-zero digits in their representation in the binary scale.
\end{theorem} 

In this work, we prove that these results are \emph{exceptional}. Namely, we show that it is possible to obtain infinite families of perfect powers (coprime with $x$) having exactly $k \ge 3$ non-zero digits in their base $x \ge 2$ representation (moreover, we will show that we can almost always provide infinite families of perfect squares) for all values of $x$ and $k$, except for the case $x=2$, $k=4$ studied by Corvaja and Zannier (Theorem \ref{4digits}). 

\section{Main result}
Consider the equation
\addtocounter{equation}{-1}
\begin{equation}\label{maincase}
y^d=c_0+\sum_{i=1}^{k-1} c_ix^{m_i}.
\end{equation} 

In this work we want to determine whether the Diophantine equation (\ref{maincase}) admits infinitely many solutions, for given values of $x$ and $k$. Arguing that some solutions can be induced from polynomial identities, and since intuitively, as the number of terms $k$ increase, we can guess that there are more and more polynomial powers $P(T)^d$ having exactly $k$ non-zero terms, our expectation is that, as $k$ increases, it is easier to find infinite families of perfect powers with exactly $k$ non-zero digits; our approach will focus on finding such families in some specific setting. Actually, we will see that finiteness results can only be obtained in the cases $k=2$ and $k=4,x=2$.

First, notice that the natural expansion of $(1+X_1+\dots+X_{p-1})^d \in \mathbb{C}[X_1,\dots,X_{p-1}]$ has exactly $\binom{p-1+d}{d}$ distinct terms. Therefore, we can choose a suitable specialization $X_i = x^{\alpha_i}$, with positive integers $\alpha_i$ such that different terms of the expansion yield different powers of $x$; under the assumption that $x$ is greater than all coefficients of this expansion, we can obtain a correspondence between the terms of this expansion and the digits of our desired perfect power, and thus obtain perfect powers whose base $x$ representation has exactly $\binom{p-1+d}{d}$ non-zero digits. Similarly, under the same assumptions, we can choose a set of exponents $\alpha_i$ such that there are exactly $\beta$ equalities among those terms, for relatively small values of $\beta$, thus
obtaining perfect powers having exactly $\binom{p+d}{d}-\beta$ non-zero digits in their base $x$ representation (where $\beta$ hopefully takes all values between $0$ and $\binom{p-1+d}{d}-\binom{p-2+d}{d-1}-1$). 

From this argument it is possible to obtain, for a fixed value of $d$, families of infinite perfect powers  having exactly $\binom{p+d}{d}-\beta$ non-zero digits in their base $x$ representation; such a construction can be done with some work (with some modifications on the arguments we will use in the next parts of this paper), remembering that $x$ has to be larger than any coefficient appearing in the expansion  $(1+X_1+\dots+X_{p-1})^d \in \mathbb{C}[X_1,\dots,X_{p-1}]$ and making sure to find suitable constructions for all values of $\beta \in [0, \ldots, \binom{p+d}{d} - \binom{p+d-1}{d} ]$.

This simple idea naturally directs us to the best case: the integers $\binom{i}{2}$ form a sequence of relatively small intervals partitioning $\mathbb{N}$, and the coefficients of the expansion of $(1+X_1+\dots+X_{p-1})^2$ are all either $1$ or $2$. For $p \ge 1$ and $0=\alpha_0 < \alpha_1 < \dots < \alpha_{p-1}$ 
we can expand $(1+X_1+\dots+X_{p-1})^2$ in the following way:
\begin{equation}\label{sqbin}
\begin{gathered}	
(x^{\alpha_0}+x^{\alpha_1}+\dots+x^{\alpha_{p-1}})^2=x^{2\alpha_0}+(2x^{\alpha_0+\alpha_1})+x^{2\alpha_1}+\left(2x^{\alpha_2+\alpha_1} + 2x^{\alpha_2+\alpha_0}\right)+x^{2\alpha_2} \tag{*} \\ +\dots+x^{2\alpha_{p-3}}+\left( \sum_{i=0}^{p-3} 2x^{\alpha_{p-2}+\alpha_i} \right)+x^{2\alpha_{p-2}}+\left(\sum_{i=0}^{p-2}2x^{\alpha_{p-1}+\alpha_i}\right)+x^{2\alpha_{p-1}}. 
\end{gathered}
\end{equation}
Clearly $x$ is always not less than all the coefficients, and if $x > 2$, this expression can be used as a starting point to yield a representation.
However, if $x=2$, this expression needs to be slightly adjusted to become a binary representation, and for this motive we might have to slightly alter our construction; thus we will discuss the case $x=2$ separately from the rest.

\subsection{Perfect powers with arbitrary number of binary digits}
Clearly, the only admissible digits in the binary scale are $0$ and $1$, thus, in base $2$, equation (\ref{maincase}) becomes $$y^d=1+2^{\alpha_1}+\dots+2^{\alpha_{k-1}}.$$

The case $k \le 4$ has been widely studied in the literature. A well-known Theorem by Mihailescu states that there is only one odd perfect power having exactly two non-zero digits, that is, $3^2=1+2^3$. Recently, Szalay (see \cite{S}) completely solved the equation $y^2 = 2^a+2^b+1$. Further, the equation $y^n = 2^a + 2^b+1$, with $n \ge 2$ has been completely solved by Bennett et al. in \cite{BBM}, thus completing the study of perfect powers having exactly $3$ non-zero binary digits. In this context, it is worth noticing that the expansion $(1+2^{\alpha_1})^2$ (which is a trivial case of our argument) yields an infinite family of perfect squares with this property - see also Theorem \ref{3digits}.

In the same work \cite{BBM}, the authors also solved completely the equation $y^n = 2^a + 2^b + 2^c + 1$ for $n \ge 5$, dealing with perfect powers having $4$ non-zero binary digits. In this context, Theorem \ref{4digits} states that there are only finitely many such perfect powers not divisible by $2$.

In this work, we will then focus on the remaining cases, assuming $k \ge 5$. Clearly, Equation (\ref{sqbin}) can be adjusted to obtain the following binary representation (remember that $\alpha_0=0$):
\begin{equation}\label{sqbin2}
\begin{gathered}	
(2^{\alpha_0}+2^{\alpha_1}+\dots+2^{\alpha_{p-1}})^2=2^{2\alpha_0}+(2^{\alpha_0+\alpha_1+1})+2^{2\alpha_1}+\left(2^{\alpha_2+\alpha_1+1} + 2^{\alpha_2+\alpha_0+1}\right)+2^{2\alpha_2} \tag{$\star$} \\ +\dots+2^{2\alpha_{p-3}}+\left( \sum_{i=0}^{p-3} 2^{\alpha_{p-2}+\alpha_i+1} \right)+2^{2\alpha_{p-2}}+\left(\sum_{i=0}^{p-2}2^{\alpha_p+\alpha_i+1}\right)+2^{2\alpha_{p-1}}. 
\end{gathered}
\end{equation}

We rearranged the expression in this way since, for $i=1,\dots,p-1$ the $i$th bracket contains pairwise distinct terms, ranging between $2^{\alpha_i+\alpha_0+1}=2^{\alpha_i+1}$ and $2^{\alpha_i+\alpha_{i-1}+1}$. Thus if $\alpha_i \ge \alpha_{i-1}+2$ every term of the $i$th bracket is strictly lower than $2^{2\alpha_i}$, while if $\alpha_i \ge 2\alpha_{i-1}-1$ then all terms of that bracket are larger than $2^{2\alpha_{i-1}}$, with equality happening if and only if $2^{\alpha_i+\alpha_0+1}=2^{2\alpha_{i-1}}$, that is, if and only if $\alpha_i = 2\alpha_{i-1}-1.$ Hence, if $\alpha_i \ge 2\alpha_{i-1}-1$, equation (\ref{sqbin2}) yields a perfect square having $\binom{p+1}{2}$ terms, with at most $p-2$ coincident terms, given by the number of indexes such that $\alpha_i = 2\alpha_{i-1}-1.$

Therefore, we can easily prove the following.
\begin{lemma}\label{2not}
	Let $k$ be a positive integer greater than $4$ not of the form $\binom{p}{2}+1$, for a positive integer $p$. Then there exist infinitely many odd perfect squares having exactly $k$ non-zero digits in their representation in the binary scale.
\end{lemma}
\begin{proof}
	Write $k$ as $k=\binom{p+1}{2}-\beta$, with $\beta \in \{0,\dots,p-2 \}$. Define a sequence $(\alpha_1,\dots,\alpha_{p-1})$ of positive integers such that
	$$\begin{cases} \alpha_1 \ge 3, \\ \alpha_i=2\alpha_{i-1}-1 \text{  for  } i=2,\dots,\beta+1, \\ \alpha_i >  2\alpha_{i-1}-1 \text{  for  } i > \beta+2.\\  \end{cases}$$
	Then, arguing as in the previous paragraphs, we can show that there are exactly $\beta$ coincident terms in the expansion (\ref{sqbin2}); moreover, those coincident terms are of the form $2^{2\alpha_{i-1}}$ and $2^{\alpha_i+\alpha_0+1}$, which then form the term $2^{2\alpha_{i-1}}+2^{\alpha_i+\alpha_0+1}=2^{\alpha_i+\alpha_0+2} < 2^{\alpha_i+\alpha_1+1}$ (since $\alpha_1 \ge 3$): thus the positive integer $y=(1+2^{\alpha_1}+\dots+2^{\alpha_{p-1}})$ is such that $y^2$ has exactly $\binom{p+1}{2}-\beta=k$ non-zero digits in its representation in the binary scale.
\end{proof}
Notice that if $k=\binom{p}{2}+1$ (i.e. $\beta=p-1$) this method would not work. Thus we have to prove this case in a slightly different way.
\begin{lemma}
	Let $k$ be a positive integer greater than $4$ of the form $\binom{p}{2}+1$, with $p$ a positive integer. Then there are infinitely many odd perfect squares having exactly $k$ non-zero digits in their binary representation.
\end{lemma}
\begin{proof}
	Notice that the binary representation of $(1+2^{\alpha_1}+2^{\alpha_1+1}+2^{\alpha_1+2})^2$ is given by
	$$(1+2^{\alpha_1}+2^{\alpha_1+1}+2^{\alpha_1+2})^2=1+2^{\alpha_1+1}+2^{\alpha_1+2}+2^{\alpha_1+3}+2^{2\alpha_1}+2^{2\alpha_1+4}+2^{2\alpha_1+5},$$
	hence it has exactly $7=\binom{4}{2}+1$ non-zero digits; while, if $k \ge 11$ define as before an infinite sequence $(\alpha_1,\dots,\alpha_{p-1})$  of positive integers such that $$\begin{cases} \alpha_1 \ge 4, \\ \alpha_i=\alpha_1+i-1 \text{  for  } i=2,3, \\ \alpha_4 =  2\alpha_1+4 \text{       } , \\ \alpha_i=2\alpha_{i-1}-1 \text{  for  } i > 4.  \end{cases}.$$
	Let $y=1+2^{\alpha_1}+2^{\alpha_2}+\dots+2^{\alpha_{p-1}}.$ Then the expansion (\ref{sqbin2}) of $y^2$ has $\binom{p+1}{2}$ terms; let us count how many equalities there are between those terms:
	\begin{itemize}
		\item There are $3$ equalities depending on $\alpha_1,\alpha_2,\alpha_3$ only, which we deduce from the binary representation of $(1+2^{\alpha_1}+2^{\alpha_2}+2^{\alpha_3})^2$ (which has $\binom{5}{2}-3=7$ non-zero digits);
		\item There are $p-4$ equalities, one for each of the  $\alpha_i$, with $i > 4$; these $\alpha_i$ are chosen so that every term of the form $2^{2\alpha_i}$ is equal to the maximum term preceding it in the expansion (\ref{sqbin2}).
	\end{itemize}
	Therefore there are exactly $p-1$ equalities, and since each of the terms obtained by adding these coincident terms is distinct from any other term of the expansion since $\alpha_1 \ge 4$, we deduce that $y^2$ has exactly $\binom{p+1}{2}-(p-1)=\binom{p}{2}+1=k$ non-zero digits in its representation in the binary scale.
\end{proof}

Combining the last two results, we obtain the following result.
\begin{theorem}
	Let $k \ge 2$ be an integer. 
	\begin{enumerate}
		\item If $k \in \{2,4\}$, then there are only finitely many odd perfect powers in $\mathbb{N}$ having precisely $k$ non-zero digits in their representation in the binary scale.
		\item If $k \not \in \{2,4\}$, then there are infinitely many odd perfect squares in $\mathbb{N}$ having precisely $k$ non-zero digits in their representation in the binary scale.
	\end{enumerate}
\end{theorem}
\subsection{Perfect powers with arbitrary number of base $x\ge 3$ digits}
Let $x \ge 3$. Determining whether the Diophantine equation $y^d=c_1x^{m_1}+c_2$ admits finitely or infinitely many solution is a very challenging open problem, studied by several authors (see for instance \cite[\S 4.4.3]{N} for results concerning this class of Diophantine equations); however, it is known that, for fixed $x \ge 2$, this equation has at most finitely many solutions in integers $0 \le c_1,c_2 < x$, $y$ coprime to $x$ and $d \ge 2$, and thus, given a fixed scale $x \ge 3$, there are at most finitely many perfect powers having exactly $k=2$ non-zero digits in their base $x$ representation.

The case $k=3$ has been studied by Bennet and Scheerer (see \cite{BS}) for certain values of $x$ (namely $x \in \{3,4,5,8,16\}$). For our purposes, it suffices to consider the expansion $(x^a+1)^2=x^{2a}+2x^a+1$ to conclude that there are infinitely many perfect squares not divisible by $x$ which base $x$ representation has exactly three non-zero digits.

Similarly, it is easy to see that the perfect cube $(x^{a}+1)^3=x^{3a}+3x^{2a}+3x^{a}+1$ has exactly four non-zero digits in its base $x$ representation; thus implying that there are infinitely many perfect cubes having exactly four non-zero digits in their base $x$ representation.

However, the examples used in the two cases $k=3,4$ cannot be used in the general case; in fact, for larger values of $d$, the coefficients of the expansion $(x^{a}+1)^d$ become very large, and since we need that $x$ is larger than all of these coefficients, for increasingly many values of $x$ this construction would not yield a base $x$ representation (as each coefficient could be associated with more than one digit). 

We approach this case similarly to the case $x=2$. Consider the expansion (fix $\alpha_0=0$) 
\begin{equation}\label{sqbinx}
\begin{gathered}	
(x^{\alpha_0}+x^{\alpha_1}+\dots+x^{\alpha_{p-1}})^2=x^{2\alpha_0}+(2x^{\alpha_0+\alpha_1})+x^{2\alpha_1}+\left(2x^{\alpha_2+\alpha_1} + 2x^{\alpha_2+\alpha_0}\right)+x^{2\alpha_2} \tag{*} \\ +\dots+x^{2\alpha_{p-3}}+\left( \sum_{i=0}^{p-3} 2x^{\alpha_{p-2}+\alpha_i} \right)+x^{2\alpha_{p-2}}+\left(\sum_{i=0}^{p-2}2x^{\alpha_{p-1}+\alpha_i}\right)+x^{2\alpha_{p-1}}. 
\end{gathered}
\end{equation}

As before, for $i=1,\dots,p-1$ the $i$th bracket contains pairwise distinct terms, ranging between $x^{\alpha_i+\alpha_0}=x^{\alpha_i}$ and  $x^{\alpha_i+\alpha_{i-1}}$. Thus if $\alpha_i \ge \alpha_{i-1}+1$ all these terms are strictly lower than  $x^{2\alpha_i}$, while if $\alpha_i \ge 2\alpha_{i-1}$ we have $\alpha_i + \alpha_{i-1} > \ldots > \alpha_i + \alpha_0 = \alpha_i \ge 2\alpha_{i-1}$, hence all the terms are strictly larger than $x^{2\alpha_{i-1}}$, with equality happening if and only if $\alpha_i=2\alpha_{i-1}$, which would imply $x^{\alpha_i+\alpha_0+1}=x^{2\alpha_{i-1}}$. Hence, if $\alpha_i \ge 2\alpha_{i-1}$, the equation (\ref{sqbinx}) gives a perfect square having exactly $\binom{p+1}{2}$ terms, and, just like we did in the case $x=2$, we can fiddle with our exponents in order to obtain the desired number of equalities (between $0$ and $p-2$). Therefore, the following result is very straightforward.

\begin{lemma}
	Let $k$ be a positive integer greater than four not of the form $\binom{p}{2}+1$, with $p$ positive integer, and let $x \ge 3$ be an integer. Then there exist infinitely many perfect squares, not divisible by $x$, having exactly $k$ non-zero digits in their base $x$ representation.
\end{lemma}
\begin{proof}
	Write $k$ as $k=\binom{p+1}{2}-\beta$, with $\beta \in \{0,\dots,p-2 \}$. Define a sequence $(\alpha_1,\dots,\alpha_{p-1})$ of positive integers (depending on $\alpha_1$) satisfying the following conditions:
	$$\begin{cases} \alpha_1 \ge 3, \\ \alpha_i=2\alpha_{i-1} \text{  for  } i=2,\dots,\beta+1 \\ \alpha_i >  2\alpha_{i-1} \text{  for  } i > \beta+2\\  \end{cases}.$$
	Then it is straightforward (arguing as in Lemma \ref{2not}) to prove that the integer $y=(1+x^{\alpha_1}+\dots+x^{\alpha_{p-1}})$ is such that $y^2$ has exactly $\binom{p+1}{2}-\beta=k$ non-zero digits in its base $x$ representation.
\end{proof}

As in the previous Section, the remaining case  $k=\binom{p}{2}+1$ is not covered by the previous construction, but requires some slight adjustements to be made, according to the value of $x$; here, we will need to split this case in three subcases.
\begin{lemma}
	Let $k \ge 7$ be an integer of the form $\binom{p}{2}+1$, for some positive integer $p$. Then there are infinitely many perfect squares not divisible by $3$ having exactly $k$ non-zero digits in their base $3$ representation.
\end{lemma}
\begin{proof}
	First, we consider some special cases:
	\begin{itemize}
		\item The perfect square $(1+3^{\alpha_1}+3^{\alpha_1+1}+3^{\alpha_1+2})^2$ has exactly $7$ non-zero digits in its base $3$ representation.
		
		\item The expansion $(1+3^{\alpha_1}+3^{\alpha_1+1}+3^{2\alpha_1}+3^{2\alpha_1+1})^2$ yields perfect squares having exactly  $11=\binom{5}{2}+1$ non-zero digits in their base $3$ representation.
	\end{itemize}
	For $k > 11$, consider a sequence of positive integers $(\alpha_1,\dots,\alpha_{p-1})$ such that
	$$\begin{cases} \alpha_1 \ge 4, \\ \alpha_2=\alpha_1+1,  \\ \alpha_i =  2\alpha_1+i-3 \text{  for  } i =3,4,\\
	\alpha_i=2\alpha_{i-1} \text{  for  } i \ge 5.  \end{cases}.$$
	Then, by taking the integer $y=1+3^{\alpha_1}+3^{\alpha_2}+\dots+3^{\alpha_p}$, notice that, for the expansion (\ref{sqbinx}) of $y^2$, the following hold:
	\begin{itemize}
		\item There are exactly four equalities between terms of (\ref{sqbinx}) depending on our choice of $\alpha_1,\alpha_2,\alpha_3,\alpha_4$, which follow from the expansion of $(1+3^{\alpha_1}+3^{\alpha_2}+3^{\alpha_3}+3^{\alpha_4})^2$ (which has exactly $11$ non-zero digits in its base $3$ representation).
		\item There are $p-5$ equalities, one for each $\alpha_i$, with $i=5,6,\dots,p-1$, following from the condition $\alpha_i=2\alpha_{i-1}$.
	\end{itemize}
	As before, these equalities are such that the terms obtained are distinct from any other term in (\ref{sqbinx}) and that each term of the expansion yields a digit in the base $3$ representation of $y^2$, which then contains exactly $\binom{p+1}{2}-4-(p-5)=\binom{p}{2}+1=k$ non-zero digits. 
\end{proof}
\begin{lemma}
	Let $k \ge 4$ be an integer of the form $\binom{p}{2}+1$, for a positive integer $p$.
	\begin{enumerate}
		\item There are infinitely many perfect squares not divisible by $4$ having exactly $k$ non-zero digits in their base $4$ representation.
		\item There are infinitely many perfect squares not divisible by $5$ having exactly $k$ non-zero digits in their base $5$ representation.
	\end{enumerate}
\end{lemma}
\begin{proof}
	\begin{enumerate}
		\item Fix $\alpha_1 \ge 2$, and define a sequence $(\alpha_1,\dots,\alpha_{p-2})$ of positive integers such that $\alpha_i > 2\alpha_{i-1}$ for every $i=2,\dots,p-2$. Take now the integer $\displaystyle y=3\cdot 4^{\alpha_{p-2}}+2\left(\sum_{i=0}^{p-3} 4^{\alpha_i}\right),$ with $\alpha_0=0$ (remember that $p \ge 3$). Then clearly $$y^2=9 \cdot 4^{2\alpha_{p-2}}+3\left(\sum_{i=0}^{p-3}4^{\alpha_{p-2}+\alpha_i+1}\right)+4\left(\sum_{i=0}^{p-3}4^{\alpha_i}\right)^2.$$
		
		Now, examining the base $4$ representation associated to the right-hand side, the first term yields exactly two non-zero digits, the second one has $p-2$ non-zero digits, while the last bracket gives exactly $\binom{p-1}{2}$ non-zero digits (by expanding the square and remembering the conditions on $\alpha_i$); further, our conditions are such that all terms appearing on the right-hand side are pairwise distinct. Thus the base $4$ representation of $y^2$ has exactly $\binom{p-1}{2}+(p-2)+2=\binom{p}{2}+1=k$ non-zero digits.
		\item Similarly, for $\alpha_1 \ge 2$,  define a sequence $(\alpha_1,\dots,\alpha_{p-2})$ of positive integers such that $\alpha_i > 2\alpha_{i-1}$ for any $i=2,\dots,p-2$, and take  $\displaystyle y=2\cdot 5^{\alpha_{p-2}}+2 \cdot 5^{\alpha_{p-3}}+\left(\sum_{i=0}^{p-4} 5^{\alpha_i}\right),$ with $\alpha_0=0$. Then $$y^2=4 \cdot 5^{2\alpha_{p-2}}+8 \cdot 5^{\alpha_{p-2}+\alpha_{p-3}}+4 \cdot 5^{2\alpha_{p-3}}+$$ $$+\left( \sum_{i=0}^{p-4}5^{\alpha_i}\right)^2+4\left( \sum_{i=0}^{p-4} 5^{\alpha_{p-2}+\alpha_{i}}\right)+4\left(\sum_{i=0}^{p-4}5^{\alpha_{p-3}+\alpha_i}\right).$$
		
		This time, examining the base $5$ representation associated to this expansion, we easily see that the first and third term yield one non-zero digit, the second one gives $2$ digits, the fourth has exactly $\binom{p-2}{2}$ non-zero digits, whence the last two have $p-3$ non-zero digits each; since all terms appearing on the right-hand side have distinct exponents, the base $5$ representation of $y^2$ has thus exactly $2(p-3)+\binom{p-2}{2}+4=\binom{p}{2}+1=k$ non-zero digits.
		
	\end{enumerate}
	
\end{proof}

\begin{lemma}
	Let $x \ge 6$ and $k \ge 4$ be integers, with $k$ having the form $\binom{p}{2}+1$, for some positive integer $p$. Then there are infinitely many perfect squares not divisible by $x$ having exactly $k$ non-zero digits in their base $x$ representation.
\end{lemma}

\begin{proof}
	
	Let $\sigma=\left \lceil \sqrt{x+1} \right \rceil$. Since $x \ge 6$, clearly $2 \sigma \le x$ and $x < \sigma^2 < 2x$; now, for $\alpha_1 \ge 2$, define a sequence $(\alpha_1,\ldots,\alpha_{p-2})$ of positive integers such that $\alpha_i > 2\alpha_{i-1}$ for all $i=2,\ldots,p-2$, and take $y=\sigma x^{\alpha_{p-2}}+x^{\alpha_{p-3}}+\ldots+x^{\alpha_1}+1$. Clearly, fixing $\alpha_0=0$, we have 
	$$y^2=\sigma^2x^{2\alpha_{p-2}}+\left( \sum_{i=0}^{p-3} 2\sigma x^{\alpha_{p-2}+\alpha_i} \right)+ \left( \sum_{i=0}^{p-3} x^{\alpha_i} \right)^2.$$
	Our choice of $\sigma$ is such that the first term of the right-hand side has exactly $2$ non-zero digits in its base $x$ representation, while the second one has exactly $p-2$ non-zero digits, and the third one has exactly $\binom{p-1}{2}$; since all powers of $x$ appearing in this expansion have distinct exponents, we immediately deduce that the base $x$ representation of $y^2$ has exactly $\binom{p-1}{2}+p=\binom{p}{2}+1=k$ non-zero digits.	
\end{proof}

We can combine all the results of this section to achieve the desired result:
\begin{theorem}
	Let $x \ge 2$ and $k \ge 3$ be integers with $(x,k) \not \in \{ (2,4), (3,4)\}$. Then there exist infinitely many perfect squares not divisible by $x$ having exactly $k$ non-zero digits in their base $x$ representation.
\end{theorem}

The previous result affirms that the known finiteness results of Mihailescu (for $k=2$) and Corvaja-Zannier (if $k=4$ and $x=2$) are the only exceptions to the general rule. However, our construction does not work in the case $x=3, k=4$; in fact, in that case it is easy to see that it is impossible to impose more than one equality among the exponents of $$(1+3^{\alpha_1}+3^{\alpha_2})^2=1+2\cdot 3^{\alpha_1}+3^{2\alpha_1}+(2\cdot 3^{\alpha_2}+2\cdot 3^{\alpha_2+\alpha_1})+3^{2\alpha_2},$$
and that in the general expansion

\begin{equation*}\label{sqbin3}
\begin{gathered}	
(3^{\alpha_0}+3^{\alpha_1}+\dots+3^{\alpha_{p-1}})^2=3^{2\alpha_0}+(2\cdot3^{\alpha_0+\alpha_1})+3^{2\alpha_1}+\left(2\cdot 3^{\alpha_2+\alpha_1} + 2\cdot 3^{\alpha_2+\alpha_0}\right)+3^{2\alpha_2} \\ +\dots+3^{2\alpha_{p-3}}+\left( \sum_{i=0}^{p-3} 2\cdot 3^{\alpha_{p-2}+\alpha_i} \right)+3^{2\alpha_{p-2}}+\left(\sum_{i=0}^{p-2}2\cdot 3^{\alpha_{p-1}+\alpha_i}\right)+3^{2\alpha_{p-1}}
\end{gathered}
\end{equation*}
at least the four terms $1=3^{2\alpha_0},2\cdot3^{\alpha_1},2\cdot3^{\alpha_{p-1}+\alpha_{p-2}},3^{2\alpha_{p-1}}$ have different exponents from the others, and thus are very hard to \emph{remove} from the final base $3$ representation that will derive from this expansion; also, from a short computation, the only perfect squares $y^2$ having exactly four non-zero digits in their base $3$ representation, for $y \le 10^7$ not divisible by $3$, are obtained for $y \in \{7, 14, 16, 17, 26, 35, 47, 68, 350, 3788\}$.

Therefore, while we were not able to reach a conclusion in this case, we think it might be interesting to ask this Question, with which we finish this work.

\begin{question}
	Determine if there are infinitely many squares not divisible by $3$ having exactly $4$ non-zero digits in their base $3$ representation.
\end{question}

\section*{Acknowledgements}
This work is part of my PhD thesis. I would like to thank my advisers, Professors Roberto Dvornicich and Umberto Zannier for their supervision, and for helpful discussions. I would also like to thank the referee for his helpful remarks and suggestions.

\end{document}